\documentclass[a4paper,11pt, leqno]{article}

\usepackage{latexsym,enumerate,graphicx}
\usepackage{amsmath,amsthm,amsopn,amstext,amscd,amsfonts,amssymb}
\usepackage[T1]{fontenc}
\usepackage{fullpage,yhmath}
\usepackage{hyperref}
\usepackage{eucal}
\usepackage{color}
\usepackage{verbatim}
\usepackage[normalem]{ulem}
\usepackage{frcursive}
\usepackage{addfont}
\addfont{OT1}{rsfs10}{\rsfs}

\newenvironment{frcseries}{\fontfamily{frc}\selectfont}{}

\numberwithin{equation}{section}

\newtheorem{theorem}{Theorem}
\newtheorem{lemma}{Lemma}

\newtheorem{proposition}{Proposition}
\newtheorem{remark}{Remark}

\newtheorem{definition}{Definition}

\newtheorem{lem}{Lemma}[section]
\numberwithin{theorem}{section}
\numberwithin{corollary}{section}
\numberwithin{lemma}{section}
\numberwithin{definition}{section}
\numberwithin{proposition}{section}
\numberwithin{remark}{section}

\newcommand{\R}{\mathbb R}
\newcommand{\N}{\mathbb N}

\newcommand{\medint}{-\kern  -,375cm\int}
\newcommand{\dint}{\displaystyle\int}

\newcommand{\inte}{\int\!\!\!\!\int}
\newcommand{\essinf}{\mathop{\mathrm{ess\;\!inf}}}

\newcommand{\di}{\mathop{\mathrm{d}\!}}
\newcommand{\loc}{\mathrm{loc}}

\newcommand{\mathfrc}[1]{\text{\textfrc{#1}}}
\newcommand{\textfrc}[1]{{\frcseries#1}}

\makeatletter
\@namedef{subjclassname@2020}{  \textup{2020} Mathematics Subject Classification}
\makeatother

\begin{document}
\title{Comparison results for a nonlocal singular {elliptic}  problem}
\author{Barbara Brandolini, Ida de Bonis, Vincenzo Ferone, Bruno Volzone}
\date{}
\maketitle

\begin{abstract}
We provide symmetrization results in the form of mass concentration comparisons for fractional singular elliptic equations in bounded domains, coupled with homogeneous external Dirichlet conditions. Two types of comparison results are presented, depending on the summability of the right-hand side of the equation. The maximum principle arguments employed in the core of the proofs of the main results offer a nonstandard, flexible alternative to the ones described in \cite[Theorem 31]{FV}. Some interesting consequences are $L^{p}$ regularity results and nonlocal energy estimates for solutions.
\end{abstract}

\section{ Introduction}

In this paper we consider the following singular nonlocal problem

\begin{equation}\label{P}
\left\{
\begin{array}{ll}
(-\Delta)^su=\dfrac{f(x)}{u^{\gamma}}\qquad&\mbox{in }\Omega
\\
u>0&\mbox{in }\Omega
\\
u=0& \mbox{on }\mathbb{R}^N\setminus\Omega.
\end{array}
\right.
\end{equation}

\noindent
Here $s \in (0,1)$, $(-\Delta)^s$ stands for the fractional Laplacian operator, $\Omega$ is a bounded, open set  in $\mathbb{R}^N$ ($N \ge 2$) with Lipschitz boundary, $\gamma>0$ and $f$ is a nonnegative summable  function.\\
Our aim is to use symmetrization techniques in order to get a comparison result between the weak solution to problem \eqref{P} and the weak solution $v$ to a symmetrized problem, defined in the ball $\Omega^\star$ centered at the origin having the same measure as  $\Omega$, which stays in the same class as the original one (namely singular and nonlocal). 

After the seminal paper by Talenti \cite{T}, it is well-known that, if  $u \in H_0^1(\Omega)$, $v\in H_0^1(\Omega^\star)$ solve
$$
\left\{
\begin{array}{ll}
-\Delta u=f \qquad & \mbox{in } \Omega\\
u=0 & \mbox{on }\partial \Omega
\end{array}
\right. \quad \mbox{and} 
\quad
\left\{
\begin{array}{ll}
-\Delta v=f^\star \qquad & \mbox{in } \Omega^\star\\
v=0 & \mbox{on }\partial \Omega^\star,
\end{array}
\right.
$$
respectively, then 
\begin{equation}\label{talenti}
    u^\star(x)\le v(x), \qquad x \in \Omega^\star
\end{equation}
(here $^\star$ stands for the Schwarz rearrangement of a function, see Section 2 for definition and properties). 
%
From \eqref{talenti} we immediately derive, for instance, that any Lebesgue norm of $u$ is bounded from above by the same Lebesgue norm of $v$. Hence, the issue of estimating the solution $u$ of a Dirichlet problem in $\Omega$ is solved once we can estimate the solution $v$ of a symmetrized problem, which actually is a one-dimensional problem and clearly much easier to handle with. 

For local operators, the approach described above has been extended through the years to uniformly elliptic equations with lower order terms, linear and nonlinear parabolic equations, non uniformly elliptic equations, and also to problems with boundary conditions other than Dirichlet.  For a survey on the power of symmetrization techniques in both Calculus of Variations and PDEs theory we refer the interested reader to \cite{Ta}. In particular, symmetrization techniques have been applied to local, singular problems like \eqref{P} when, on the left-hand side, the Laplacian operator replaces the fractional one (see \cite{BCT}).

In the framework of nonlocal problems, the effect of symmetrization on fractional elliptic problems has  been investigated for the first time in \cite{DV} in a somewhat indirect way. Indeed, there it is used in an essential way the fact that a nonlocal problem involving the fractional Laplacian $(-\Delta)^{s}$, $s\in (0,1)$, can be linked to a suitable,
local extension problem, whose solution $\psi(x,y)$, an $s-$harmonic extension of the solution $u$ to the nonlocal problem, is defined on the infinite cylinder $\mathcal C_\Omega=\Omega\times (0,+\infty)$, to which classical symmetrization techniques (with respect to the variable $x \in \Omega$) can be applied. {For other results concerning the Neumann problems and the nonlocal Gaussian symmetrization see \cite{V}, \cite{FSV}, while symmetrization results for fractional parabolic equations of porous medium type have been achieved in \cite{SVV,VV1,VV2,V}. Moreover, we wish to mention here \cite{FVnon}, where the case of a fractional nonlinear problem is considered, and \cite{GON}, where a comparison type result in terms of the $L^{p}$ norms of solutions (thus weaker than the mass concentration comparison) is established for  
solutions to equations involving a nonlocal operator with integral kernels, hence not covering the fractional Laplacian.}
\\[0.2pt]
 In this note we adopt the direct approach introduced in \cite{FV}, where the authors deal with problem \eqref{P} in the case $\gamma=0$. This direct approach does not employ the local interpretation of the fractional Laplacian described above, while it makes a clever use of a nonlocal version of the classical P\'olya-Szeg\H{o} inequality, plus a sophisticated representation of the fractional Laplacian of a spherical mean function in $(N+2)$ dimensions, which allows to conclude by a maximum principle argument. \\ {The real novelty of this paper is that} the above mentioned interpretation is \emph{avoided} in the proofs of our new results, thus in this sense they offer an alternative to the crucial part of the proof of \cite[Theorem 31]{FV}. Furthermore, such a new technique seems to be rather flexible to be used in a broad variety of related contexts. 
 \noindent Our main results are the following theorems.
\begin{theorem}\label{mainTheorem1}
Let $s\in(0,1)$, $N\geq 2$, $\gamma>0$ and assume that $f\in L^{\infty}(\Omega)$, $f\geq 0$. If $u$ is the weak solution to problem \eqref{P} and $v$ is the weak solution to the  symmetrized problem 
\begin{equation}\label{probSim}
\left\{
\begin{array}{ll}
(-\Delta)^sv=\dfrac{||f||_{L^\infty(\Omega^{\star})}}{v^{\gamma}}\qquad&\mbox{in }\Omega^{\star}
\\
v>0& \mbox{in }\Omega^{\star}
\\
v=0& \mbox{on }\mathbb{R}^N\setminus\Omega^{\star},
\end{array}
\right.
\end{equation}
then
\begin{equation}\label{compSim}
 \int_{B_r(0)}u^\star(x)\,\di x\leq \int_{B_r(0)}v^\star(x)\,\di x,\qquad r>0.
 \end{equation}
\end{theorem}

\noindent In order to obtain some regularity results depending on the value of $\gamma$ and on the summability of $f$,  we will also prove the following comparison result.

\begin{theorem}\label{mainTheorem2}
Let $s\in(0,1)$, $N\geq 2$, $\gamma>0$ and assume that $f\in L^1(\Omega)$, $f\geq 0$. If $u$ is the weak solution to problem \eqref{P} and $v$ is the weak solution to the following problem 
\begin{equation}\label{probSimgamma}
\left\{
\begin{array}{ll}
(-\Delta)^sv=(\gamma+1)f^\star\qquad & \mbox{in }\Omega^\star
\\ \\
v=0 & \mbox{on } \mathbb{R}^N\setminus\Omega^\star,
\end{array}
\right.
\end{equation}
then 
\begin{equation}\label{compSim2}
 \int_{B_r(0)}u^\star(x)^{\gamma+1}\,\di x\leq \int_{B_r(0)}v^\star(x)\,\di x, \qquad r> 0.
 \end{equation}
\end{theorem}

We stress that analogous estimates in the local case are proved in \cite{BCT}. For example, in the same reference, instead of \eqref{compSim}, a comparison result between mass concentrations of $u^\star(x)^{\gamma}$ and $v^\star(x)^{\gamma}$ is proved. It turns out that our result slightly improves the quoted ones when $\gamma>1$ (see Remark \ref{remgamma} for details).

Moreover, some rather simple modifications to our arguments allow us to get comparison results for a singular fractional elliptic equation with a zero order term posed in $\Omega$, of the form
\[
(-\Delta)^su+ cu=\dfrac{f}{u^{\gamma}},
\]
for some bounded coefficient $c\geq0$, complemented with exterior Dirichlet boundary coefficients.
\\

\noindent The paper is organized as follows. In Section 2 we provide the functional setting of the problem and we recall some basic notion about rearrangements. Section 3 is devoted to the proof of Theorem \ref{mainTheorem1}. In Section 4 we prove Theorem \ref{mainTheorem2}, which is the key ingredient to prove the regularity results contained in Section 5. 

\section{Notation and preliminaries}

\subsection{Functional setting}

Let $s \in (0,1)$. For any open set $\Omega$ and any measurable function $u$ on $\Omega$, we introduce the fractional Gagliardo seminorm
\[
[u]_{H^{s}(\Omega)}=\left(\inte_{\Omega\times \Omega}\frac{|u(x)-u(y)|^{2}}{|x-y|^{N+2s}}\,\di x\di y\right)^{1/2}.
\]
Then we define the fractional Sobolev space $H^{s}(\Omega)$ as the space
\[
H^{s}(\Omega)=\left\{u\in L^{2}(\Omega):\,[u]_{H^{s}(\Omega)}<\infty\right\},
\]
endowed with the norm
\[
\|u\|_{H^{s}(\Omega)}=\| u \|_{L^{2}(\Omega)}+[u]_{H^{s}(\Omega)}.
\]
We denote by $H_{0}^{s}(\Omega)$ the closure of $C_{c}^{\infty}(\Omega)$ in the $H^{s}(\Omega)$ topology. Moreover, we will define the space 
\[
H_{\mathrm{loc}}^{s}(\Omega)=\left\{u:\Omega\rightarrow \R:\,u_{| K}\in H^{s}(K),\,\text{for all } K \subset\subset \Omega\right\}. 
\]


There is a strict connection between the space $H^s(\R^N)$ and the fractional Laplacian operator $(-\Delta)^s$. 
For any $s\in(0,1)$ and  $u \in$ {\rsfs S} (the classical Schwartz class), the fractional Laplacian operator is defined as
$$
(-\Delta)^s u=\gamma(N,s)\, \mathrm{P.V.} \int_{\R^N}\frac{u(x)-u(y)}{|x-y|^{N+2s}}\,\di y,
$$
where
\begin{equation}\label{gamma}
\gamma(N,s)=\frac{s\, 2^{2s}\Gamma\left(\frac{N+2s}{2}\right)}{\pi^{\frac N 2} \Gamma(1-s)}.
\end{equation}
The following result can be found in \cite[Proposition 3.6]{DNPALVAL}.
\begin{proposition}
Let $s\in (0,1)$ and $u \in H^s(\R^N)$. Then
\begin{equation*}
[u]_{H^s(\R^N)}^2=\frac{2}{\gamma(N,s)}||(-\Delta)^{\frac s 2 } u||_{L^2(\R^N)}^2.
\end{equation*}
\end{proposition}

The analytic theory of the fractional Laplacian in the whole $\R^N$ is nowadays considered classical and we refer the interested reader for example to \cite{S}.

We are interested in Dirichlet problems defined in bounded domains. To this aim, we  consider the space $X_0^s(\Omega)$, defined as 
$$
X_0^s(\Omega)=\left\{u \in H^s(\R^N): \> u=0 \mbox{ a.e. in } \R^N \setminus \Omega\right\}.
$$
When $\Omega$ is an open bounded set with Lipschitz boundary, it can be proved that (see \cite[Proposition B.1]{BrasParSquas}) $X_0^s(\Omega)$ coincides with the completion of $C_{c}^{\infty}(\Omega)$ with respect to the seminorm $ [\cdot]_{H^{s}(\R^{N})}$.  Moreover, when $2s\neq1$, it can also proved that $X_0^s(\Omega)$ coincides with $H_{0}^{s}(\Omega)$ (see \cite[Proposition B.1]{Brasco}), while in general for $2s=1$, we have a strict inclusion
 $$X_0^s(\Omega)\subset H_{0}^{s}(\Omega)$$ 
(see \cite[Remark 2.1]{secondeigenbrasco}). Indeed, we have that $X_0^s(\Omega)$ coincides with the interpolation space $H^{1/2}_{00}(\Omega)$ (see \cite[Appendix]{SirBonfVaz}).
\\
A consequence of fractional Poincar\'e inequalities (see \cite[Lemma 2.4]{Brasco}) is that we can equip the space $X_0^s(\Omega)$ with the Gagliardo seminorm
\[
\|u\|_{X_{0}^{s}(\Omega)}=[u]_{H^{s}(\R^{N})}=\left(\inte_{\R^{2N}}\frac{|u(x)-u(y)|^{2}}{|x-y|^{N+2s}}\,\di x\di y\right)^{1/2}.
\]

\noindent From the definition of $X_{0}^{s}(\Omega)$ it easily follows that for each $u\in X_{0}^{s}(\Omega)$



\begin{equation*}\label{norm2}
||u||_{X_0^s(\Omega)} = \left(\inte_{Q}\frac{|u(x)-u(y)|^2}{|x-y|^{N+2s}}\di x\di y\right)^{\frac 1 2}
\end{equation*}
where $Q=\R^{2N}\setminus\left(\right.${\rsfs C}$\Omega \times${\rsfs C}$\left.\Omega\right)$ and {\rsfs C}$\Omega=\R^N\setminus\Omega$.\\
Then we consider the \emph{restricted} fractional Laplacian $(-\Delta|_{\Omega})_{rest}$ on $\Omega$, defined by duality on the space $X_{0}^{s}(\Omega)$. Since there will be no matter of confusion, we shall keep the classical notation $(-\Delta)^{s}$ for such operator. Moreover, denoted by $X^{-s}(\Omega)$ its dual, the operator $$(-\Delta)^s: X_0^s(\Omega) \to X^{-s}(\Omega)$$ is continuous. \\
Finally, we recall that the following fractional Sobolev embedding holds true, see for instance \cite[Theorem 6.5]{DNPALVAL}.

\begin{theorem}
Let $s \in (0,1)$ and $N>2s$. There exists a positive constant $\mathcal{S}(N,s)$ such that, for any measurable and compactly supported function $u:\R^N \to \R$, it holds
\begin{equation*}
||u||_{L^{2_s^\ast}(\R^N)}^2 \le \mathcal{S}(N,s)\inte_{\R^{2N}} \frac{|u(x)-u(y)|^2}{|x-y|^{N+2s}}\,\di x\di y,
\end{equation*}
where
$$
2_s^\ast=\frac{2N}{N-2s}
$$
is the critical Sobolev exponent. In particular, if $u \in X_0^s(\Omega)$, we have
\begin{equation*} \label{sobolev}
||u||_{L^{2_s^\ast}(\Omega)}^2 \le \mathcal{S}(N,s)\inte_{\R^{2N}}\frac{|u(x)-u(y)|^2}{|x-y|^{N+2s}}\,\di x\di y.
\end{equation*}
\end{theorem}

We end this subsection with an inequality that will turns out to be very useful in the sequel. We recall that, when we deal with fractional derivatives, the chain rule does not hold true. It can be replaced by an inequality where a convex or concave function is involved (see \cite[Proposition 4]{LPPS}) and \cite[Lemma 3.3]{secondeigenbrasco}.

\begin{proposition}
Assume that $\Phi: \R \to \R$ is a Lipschitz continuous, convex function, such that $\Phi(0)=0$. Then, if $u \in X_0^s(\Omega)$, we have
\begin{equation*}
(-\Delta)^s \Phi(u)\le \Phi'(u)(-\Delta)^s u \qquad \mbox{weakly in } \Omega,
\end{equation*}
in the sense that for all nonnegative $\varphi\in X_{0}^{s}(\Omega)$ we have
\begin{align}\label{leonori}
\inte_{\R^{2N}} &\frac{[\Phi(u(x))-\Phi(u(y))][\varphi(x)-\varphi(y)]}{|x-y|^{N+2s}}\di x\di y
\\
&\leq \inte_{\R^{2N}} \frac{[u(x)-u(y)][\Phi^{\prime}(u(x))\,\varphi(x)-\Phi^{\prime}(u(y))\,\varphi(y)]}{|x-y|^{N+2s}}\,\di x\di y. \notag
\end{align}
Analogously, if $\Psi: \R \to \R$ is a Lipschitz continuous, concave function, such that $\Psi(0)=0$, and $u \in X_0^s(\Omega)$, then
\begin{equation*}\label{leonori1}
(-\Delta)^s \Psi(u)\ge \Psi'(u)(-\Delta)^s u \qquad \mbox{weakly in } \Omega.
\end{equation*}

\end{proposition}

\subsection{Schwarz symmetrization}

We now recall some notions about Schwarz symmetrization and some related fundamental pro-\\perties. For more details we refer the interested reader, for instance,  to \cite{Bandle,BS,Kes,Talenti}.

\noindent 
Let $u$ be a real measurable function on $\Omega$. If $u$ is such that its \textit{distribution function} $\mu_u$ satisfies 
$$\mu_u(t):=|\{x\in\Omega: |u(x)|>t\}|<+\infty, \quad \mbox{for every}\,\,t>0,$$
%
we define the decreasing rearrangement of $u$ as the generalized inverse of $\mu_u$, that is
$$
u^\ast(\sigma)=\sup\{t\ge 0: \mu_u(t)>\sigma\}, \quad \sigma>0.
$$
The radially symmetric, decreasing rearrangement of $u$, also known as  the Schwarz decreasing rearrangement of $u$, is hence defined as 
$$u^{\star}(x)=u^\ast(\omega_N|x|^N)
\qquad x \in \Omega^\star,
$$
where $\omega_N$ is the measure of the unitary ball in $\R^N$, and $\Omega^\star$ is the ball (centered at the origin) having the same measure as $\Omega$.
From the definitions given above we can easily deduce that $u$, $u^\ast$ and $u^\star$ are equi-distributed, that is 
$$
\mu_u=\mu_{u^\ast}=\mu_{u^\star}.
$$
Moreover, the following properties hold true.

\begin{proposition}
 Let $u,v:\Omega \to \mathbb{R}$ be two measurable functions
satisfying
\begin{equation}  \label{uv}
\mu_u(t)<\infty , \quad \mu_v(t)<\infty, \quad \mbox{for every } t>0.
\end{equation}
Then 
\begin{itemize}
\item[(i)] if $|v|\le |u|$ a.e., then $v^\ast \le u^\ast$; 

\item[(ii)]   $(c\, u)^\ast=|c|\, u^\ast$, for every $c \in \mathbb{R}$;

\item[(iii)]  if $H:[0,\infty] \to [0,\infty]$ is an increasing,
continuous function, then $H(|u|)^\ast=H\left(u^\ast\right)$;

\item[(iv)]  if $u\in L^p(\Omega)$, $1\leq p\leq \infty$, then $u^\ast \in L^p(0,|\Omega|)$, $u^\star\in L^p(\Omega^\star)$ and $$||u||_{L^p(\Omega)}=||u^\ast||_{L^p(0,|\Omega|)}=||u^{\star}||_{L^p(\Omega^{\star})}.$$

 \end{itemize}
\end{proposition}

\noindent Furthermore, the celebrated Hardy-Littlewood inequality holds true
\begin{equation}\label{hl}
\int_{\Omega} |u(x)v(x)|\,\di x\leq \int_0^{|\Omega|}u^\ast(r)v^\ast(r)\,\di r = \int_{\Omega^{\star}}u^{\star}(x)v^{\star}(x)\,\di x.
\end{equation}

\noindent We can also define the maximal function of the rearrangement of $u$
$$
u^{\ast\ast}(\sigma)=\frac 1 \sigma \int_0^\sigma u^\ast(t)\, \di t, \qquad \sigma >0.
$$ 
It is easy to prove (see, for example, \cite{BS}) that the Lebesgue norms of $u^\ast$ and $u^{\ast\ast}$ are equivalent, that is there exists $C>0$ such that
\begin{equation*}\label{ustarstar}
||u^\ast||_{L^p(0,|\Omega|)}\le ||u^{\ast\ast}||_{L^p(0,|\Omega|)}\le C ||u^\ast||_{L^p(0,|\Omega|)}.
\end{equation*}

Since we will prove comparison results between integrals of solutions to nonlocal problems, the following definition will play a fundamental role.

\begin{definition}
 Let $u,v\in L^1_{\mathrm{loc}}(\R^N)$.
 We say that $u$ is less concentrated than $v$, and we write $u \prec v$, if for every $\sigma>0$ we have
 $$
 \int_0^\sigma u^\ast(t)\, \di t \le \int_0^\sigma v^\ast(t)\, \di t,
 $$
 or, equivalently, for every $r>0$,
 $$
 \int_{B_r(0)}u^\star(x)\, \di x\le \int_{B_r(0)}v^\star(x)\, \di x.
 $$
\end{definition}
\noindent Clearly, this definition can be adapted to functions defined in an open subset $\Omega$ of $\R^N$, by extending the functions to zero outside $\Omega$. 
The partial order relationship $\prec$ is called comparison of mass concentrations and it satisfies some nice properties (see \cite{ALT}).

\begin{proposition}\label{Propconves}
Let $u,v \in L^1(\Omega)$ be two nonnegative functions. Then, the following statements are equivalent:
\begin{itemize}
\item[(a)] $u \prec v$;
\item[(b)] for all nonnegative $\varphi\in L^{\infty}(\Omega)$
$$\int_\Omega u(x)\varphi(x)\, \di x\leq \int_0^{|\Omega|} v^\ast(r)\varphi^\ast(r)\, \di r = \int_{\Omega^\star}v^\star(x)\varphi^\star(x)\, \di x;$$
\item[(c)] for all convex, nonnegative, Lipschitz function $\Phi$, such that $\Phi(0)=0,$
$$\int_{\Omega}\Phi(u(x))\,\di x\leq \int_{\Omega}\Phi(v(x))\,\di x.$$
\end{itemize}
\end{proposition}

\noindent From Proposition \ref{Propconves} we immediately deduce that, if $u \prec v$, then
$$
||u||_{L^p(\Omega)}\le ||v||_{L^p(\Omega)}, \qquad 1 \le p \le +\infty.
$$
Moreover, if $u,\,v\in L^{p}(\Omega)$ with $p>1$, the inequality in point $(b)$ above holds for all nonnegative $\varphi\in L^{p^{\prime}}(\Omega)$.
\\
We end this subsection by recalling the following generalization of the Riesz rearrangement inequality (see, for example, \cite[Theorem 2.2]{AL}).

\begin{proposition}
Let $F: \R^+\times \R^+ \to \R^+$ be a continuous function such that $F(0,0)=0$ and
$$
F(u_1,v_1)+F(u_2,v_2)\ge F(u_1,v_2)+F(u_2,v_1)
$$
whenever $u_2 \ge u_1>0$ and $v_2 \ge v_1>0$. Assume that $u,v$ are two nonnegative, measurable functions on $\R^N$ satisfying \eqref{uv}. Then
\begin{equation}\label{riesz}
\inte_{\R^{2N}} F(u(x),v(y))\, W(ax+by)\,\di x \di y \le \inte_{\R^{2N}} F(u^\star(x),v^\star(y))\, W(ax+by)\,\di x \di y
\end{equation}
and
$$
\int_{\R^N}F(u(x),v(x))\,\di x \le \int_{\R^N} F(u^\star(x),v^\star(x))\, \di x,
$$
for any nonnegative $W \in L^1(\R^N)$ and any choice of nonzero numbers $a,b$.
\end{proposition}

\vspace{0.3cm}

\subsection{Two fundamental lemmata} 

We start by proving the following result, which will be fundamental in the crucial maximum principle arguments established in the proofs of Theorem \ref{mainTheorem1}  and Theorem \ref{mainTheorem2}. It is based on a technique introduced in \cite[Theorem 1]{AVV} and subsequently used in \cite[Theorem 3.2]{VV1}.

\begin{lem}\label{MaxMin}
Let $u,v$ be two nonnegative, continuous functions on $[0,R]$. Let us define
\begin{equation*}\label{defH}
H_u(r)=\int_0^ru(\rho)\rho^{N-1}\,\di \rho\quad\quad\quad H_v(r)=\int_0^rv(\rho)\rho^{N-1}\,\di\rho
\end{equation*}
and 
\begin{equation*}\label{defK}
K_u(r)=\int_r^R u(\rho)\rho^{N-1}\,\di\rho\quad\quad\quad K_v(r)=\int_r^Rv(\rho)\rho^{N-1}\,\di\rho.
\end{equation*}
Assume that $H_u(r)-H_v(r)$ admits a positive maximum point at $\bar r >0$, that is, 
\begin{equation*}\label{massimoH}
0<H_u(\bar{r})-H_v(\bar{r})=\max_{r\in [0,R]}(H_u(r)-H_v(r)).
\end{equation*}
Then, if $h\not\equiv 0$ is a positive, increasing bounded function on $(0,R)$, we have
\begin{equation}\label{increasing}
\int_0^{\bar{r}}u(\rho)h(\rho)\rho^{N-1}\,\di\rho-\int_0^{\bar{r}}v(\rho)h(\rho)\rho^{N-1}\,\di\rho>0.
\end{equation}
Analogously, assume that $K_u(r)-K_v(r)$ admits a negative minimum point at $\bar{r}<R$, that is,
\begin{equation*}\label{minimoK}
0>K_u(\bar{r})-K_v(\bar{r})=\min_{r\in [0,R]}(K_u(r)-K_v(r)).
\end{equation*}
Then, if $h\not\equiv 0$ is a positive, decreasing bounded function on $(0,R)$, we have
\begin{equation}\label{decreasing}
\int_{\bar{r}}^{R}u(\rho)h(\rho)\rho^{N-1}\,\di\rho-\int_{\bar{r}}^{R}v(\rho)h(\rho)\rho^{N-1}\,\di\rho<0.
\end{equation}
\end{lem}

\proof
It is enough to observe that \eqref{increasing} is an immediate consequence of the following integration by parts
\begin{eqnarray*}
&&\int_{0}^{\bar {r}}u(\rho)h(\rho)\rho^{N-1}\,\di\rho- \int_{0}^{\bar {r}}v(\rho)h(\rho)\rho^{N-1}\,\di\rho
\\
&&=h(0)\big(H_u(\bar {r})-H_v(\bar {r})\big)
+\int_{0}^{\bar {r}}\Big[\big(H_u(\bar{r})-H_v(\bar{r})\big)-\big(H_u(\rho)-H_v(\rho)\big)\Big]\,\di h(\rho)>0.
\end{eqnarray*}
Analogously, concerning \eqref{decreasing}, we have
\begin{eqnarray*}
&&\int_{\bar{r}}^{R}u(\rho)h(\rho)\rho^{N-1}\,\di\rho-\int_{\bar{r}}^{R}v(\rho)h(\rho)\rho^{N-1}\,\di\rho
\\
&&=h(R)\big(K_u(\bar{r})-K_v(\bar{r})\big)
-\int_{\bar{r}}^{R}\Big[\big(K_u(\bar {r})-K_v(\bar{r})\big)-\big(K_u(\rho)-K_v(\rho)\big)\Big]\,\di h(\rho)<0.
\end{eqnarray*}
\endproof

\begin{remark}\label{hk1}
We explicitly observe that, reasoning as in the proof of Lemma \ref{MaxMin}, we can prove that, if $\max_{r\in [0,R]}(H_u(r)-H_v(r))=0$, then 
\begin{equation*} 
\int_0^{\bar{r}}u(\rho)h(\rho)\rho^{N-1}\,\di\rho-\int_0^{\bar{r}}v(\rho)h(\rho)\rho^{N-1}\,\di\rho\ge 0.
\end{equation*}
Analogously, if $\min_{r\in [0,R]}(K_u(r)-K_v(r))=0$, then
\begin{equation*} 
\int_{\bar{r}}^{R}u(\rho)h(\rho)\rho^{N-1}\,\di\rho-\int_{\bar{r}}^{R}v(\rho)h(\rho)\rho^{N-1}\,\di\rho\le0.
\end{equation*} 
\end{remark}

\begin{remark}\label{hk}
If $\bar r\in (0,R)$ is a maximum point for $H_u(r)-H_v(r)$, then $\bar r$ is a non-positive minimum point for $K_u(r)-K_v(r)$. 
Indeed, it is easy to see that
\begin{eqnarray*}
K_u(r)-K_v(r)&=&\left(H_u(R)-H_v(R)\right)-\left(H_u(r)-H_v(r)\right)
\\
&\ge&\left(H_u(R)-H_v(R)\right)-\left(H_u(\bar r)-H_v(\bar r)\right)=K_u(\bar r)-K_v(\bar r). 
\end{eqnarray*}
Being $\bar{r}$ a maximum point for $H_u(r)-H_v(r)$ we have $K_u(\bar r)-K_v(\bar r)\leq 0$ and, by the above inequality,
$$\min_{r\in[0,R]}\left(K_u(r)-K_v(r)\right)=K_u(\bar r)-K_v(\bar r)\leq 0.$$

\end{remark}

\begin{lemma}\label{lemmaab}
Let $\gamma>0$. Then, for every $a,b>0$, we have
\begin{equation}\label{ab}
\frac{1}{a^\gamma}-\frac{1}{b^\gamma}\le \frac{\gamma}{a^{\gamma+1}}(b-a).
\end{equation}
\end{lemma}

\proof
It is immediate to check that, setting $g(t)=t^\gamma+\frac \gamma t$, 
$$
\min_{t>0}g(t)=g(1)=\gamma+1.
$$
Choosing $t=\frac{a}{b}$ we get the claim.
\endproof


\subsection{A key function}

We end this section by discussing some properties of the function
$$
\Theta_{N,s}(r,\rho)=\frac{1}{N\omega_N} \int_{|x'|=1}\left(\int_{|y'|=1}\frac{1}{|r\,x'-\rho\, y'|^{N+2s}}\, \di \mathcal{H}^{N-1}(y')\right)\di \mathcal{H}^{N-1}(x')
$$
defined for $r,\rho>0$.
We observe that the internal integral does not depend on $x'$. So we can compute it by choosing any fixed $x'$ and we obtain
\begin{eqnarray}
\Theta_{N,s}(r,\rho)&=&\int_{|y'|=1}\frac{1}{|r\,x'-\rho\, y'|^{N+2s}}\, \di \mathcal{H}^{N-1}(y') \notag
\\
&=& \frac{2\pi^{\frac{N-1}{2}}}{\Gamma\left(\frac{N-1}{2}\right)}\int_0^\pi \frac{\sin^{N-2}\theta}{(r^2-2r\rho\cos \theta+\rho^2)^{\frac{N+2s}{2}}}\,\di \theta. \label{theta}
\end{eqnarray}
Identity \eqref{theta} immediately infers that $\Theta_{N,s}(r,\rho)$ is symmetric, that is 
$$
\Theta_{N,s}(r,\rho)=\Theta_{N,s}(\rho,r), \qquad r,\rho >0.
$$
Moreover,
\begin{equation}\label{hyper}
\Theta_{N,s}(r,\rho)=\left\{
\begin{array}{ll}
\dfrac{2\pi^{\frac{N-1}{2}}}{\Gamma\left(\frac{N-1}{2}\right)}\rho^{-N-2s}  {}_{2}F_1\left(\frac{N+2s}{2},s+1;\frac N 2; \frac{r^2}{\rho^2}\right) \quad & \mbox{if } 0 \le r<\rho<+\infty
\\ \\
\dfrac{2\pi^{\frac{N-1}{2}}}{\Gamma\left(\frac{N-1}{2}\right)}r^{-N-2s} {}_2F_1\left(\frac{N+2s}{2},s+1;\frac N 2; \frac{\rho^2}{r^2}\right) & \mbox{if } 0 \le \rho<r<+\infty,
\end{array}
\right.
\end{equation}
where ${}_2 F_1 (a,b;c;x)$ is the hypergeometric function (see, for example, \cite[Ch. 9]{L}) defined by
$$
{}_2 F_1 (a,b;c;x)=\frac{\Gamma(c)}{\Gamma(b)\Gamma(c-b)}\int_0^1 \tau^{b-1}(1-\tau)^{c-b-1}(1-x\tau)^{-a}\,\di \tau, \quad c>b>0, 0<\tau<1.
$$
It is well-known that
$$
{}_2 F_1' (a,b;c;x)=\frac{ab}{c}{}_2 F_1 (a+1,b+1;c+1;x).
$$
Hence we immediately get that, 
 if $\bar r>0$, $\Theta_{N,s}(r,\rho)$ is  increasing with respect to $r\in [0,\bar r]$ for any fixed $\rho>\bar r$, while it is decreasing with respect to $\rho >\bar r$ for any fixed $r \in [0,\bar r]$.

\noindent Finally, using \eqref{hyper}, we have the following asymptotic behaviors: 
\begin{equation}\label{asym}
\Theta_{N,s}(r,\rho) \sim \frac{1}{|r-\rho|^{1+2s}} \qquad \mbox{as } |r-\rho|\to 0.
\end{equation}
and
\begin{equation}\label{asym1}
\Theta_{N,s}(r,\rho) \sim \frac{1}{r^{N+2s}} \quad \mbox{as } r \to +\infty, \quad \Theta_{N,s}(r,\rho) \sim \frac{1}{\rho^{N+2s}} \quad \mbox{as } \rho \to +\infty.
\end{equation}


\section{Proof of Theorem \ref{mainTheorem1}}

Before proving our main result we need to specify the notion of solution to problem \eqref{P}. Note that, due to the lack of regularity of solutions near the boundary, the notion of solution has to be understood in the weak distributional meaning, for test functions compactly supported in the domain. Furthermore, the nonlocal nature of the operator has to be taken into account.\\
We will adopt the following notion of solution contained in  \cite{CMSS}.

\begin{definition}
We say that a positive function  $u\in H^s_{\loc}(\Omega)\cap L^1(\Omega)$ is a weak solution to problem \eqref{P} if 
$$u^{\max\{\frac{\gamma+1}{2},1\}}\in X_0^s(\Omega),\quad\quad \frac{f}{u^{\gamma}}\in L^1_\loc(\Omega),$$
and, for every nonnegative $\varphi\in C_c^{\infty}(\Omega)$, we have
\begin{equation*}\label{energySol}
\frac{\gamma(N,s)}{2}\inte_{\R^{2N}}\frac{(u(x)-u(y))(\varphi(x)-\varphi(y))}{|x-y|^{N+2s}}\,\di x\di y=\int_{\Omega}\frac{f(x)}{u(x)^{\gamma}}\varphi(x)\,\di x,
\end{equation*}
with $\gamma(N,s)$ defined in \eqref{gamma}.
\end{definition}
In \cite[Theorem 1.2]{CMSS} (see also \cite{BDMP}), the authors prove the existence of a weak solution to problem \eqref{P} with $\essinf_K u>0$ for every compact set $K \subset \subset \Omega$, distinguishing two cases according to the value of $\gamma$: 1) (mildly singular)  when $0<\gamma\le 1$ and $f \in L^p(\Omega)$, then there exists a solution $u \in X_0^s(\Omega)$; 2) (strongly singular) when $\gamma>1$ and $f \in L^1(\Omega)$, then there exists a solution $u \in H^s_\loc(\Omega)\cap L^1(\Omega)$ such that $u^{\frac{\gamma+1}{2}}\in X_0^s(\Omega)$. In the same paper the authors also discuss the uniqueness of such solutions. Since the way of understanding the boundary condition is not unambiguous, they start with the following:

\begin{definition} Let $u$ be such that $u =0$ in $\R^N\setminus \Omega$. We say that $u \le 0$ on $\partial \Omega$ if, for every $\varepsilon >0$, it follows that
$$
(u-\varepsilon)_+\in X_0^s(\Omega).
$$
We will say that $u=0$ on $\partial \Omega$ if $ u$ is nonnegative and $ u \le 0$ on $\partial \Omega$.
\end{definition}
\noindent Adopting such a definition, in \cite[Theorem 1.4]{CMSS} the authors also show if $\gamma>0$ and $f \in L^1(\Omega)$, there exists at most one weak solution to problem \eqref{P}.

\noindent We can finally prove Theorem \ref{mainTheorem1}.\medskip

\noindent{\sc Proof of Theorem \ref{mainTheorem1}.}\\
We split the proof into different steps. \\

\vspace{0.2cm}
\noindent \emph{ Step 1. Approximating problems}

\smallskip

For every $k\in \mathbb{N}$ let us define 
$$f_k:=\min\left\{f(x),k\right\}$$
and let us consider the following sequence of nonsingular approximating problems 
\begin{equation}\label{pk}
\left\{
\begin{array}{ll}
(-\Delta)^s u_k =\dfrac{f_k}{(u_k+\frac{1}{k})^{\gamma}}\qquad& \mbox{in }\Omega
\\
u_k>0& \mbox{in }\Omega
\\
u_k=0& \mbox{on }\mathbb{R}^N\setminus\Omega.
\end{array}
\right.
\end{equation}
For every $k \in \N$ problem \eqref{pk} has a nonnegative solution belonging to $X_0^s(\Omega)\cap L^{\infty}(\Omega)$ (see \cite[Lemma 3.1]{BDMP}), which means that 
\begin{equation}\label{weakapprox}
\frac{\gamma(N,s)}{2}\inte_Q\frac{(u_k(x)-u_k(y))(\varphi(x)-\varphi(y))}{|x-y|^{N+2s}}\,\di x\di y=\int_{\Omega}\frac{f_k(x)}{\left(u_k(x)+\frac{1}{k}\right)^{\gamma}}\varphi(x)\,\di x
\end{equation}
for every $\varphi\in X_0^s(\Omega)$. Moreover, the sequence $u_k$ is increasing, $u_k>0$ in $\Omega$, and, for every subset $\omega \subset\subset \Omega$, there exists a positive constant $c_\omega$, independent of $k$, such that $u_k (x)\ge c_\omega>0$ for every $x \in \omega$ and $k \in \N$ (see \cite[Lemma 3.2]{BDMP}).

\vspace{0.2cm}

\noindent \emph{Step 2. Reduction to the radial case}

\smallskip
\noindent Let $0\leq t<||u_k||_{L^\infty(\Omega)}$ and $h>0$. We consider the following test function 
$$\varphi(x)=\mathcal{G}_{t,h}(u_k(x)),$$ 
where $\mathcal{G}_{t,h}(\theta)$ is defined as follows:
$$
\mathcal{G}_{t,h}(\theta)=\begin{cases}
h \qquad & \mbox{if } \theta>t+h\\
\theta-t & \mbox{if } t<\theta\leq t+h\\
0  & \mbox{if } \theta\leq t.
\end{cases}
$$

\noindent We explicitly observe that $\mathcal{G}_{t,h}(\theta)\in X_0^s(\Omega)$, so we can use it in the weak formulation of solution \eqref{weakapprox}, obtaining

\begin{eqnarray}\label{b}
&&\frac{\gamma(N,s)}{2}\inte_{\R^{2N}}\frac{[u_k(x)-u_k(y)]\left[\mathcal{G}_{t,h}(u_k(x))-\mathcal{G}_{t,h}(u_k(y))\right]}{|x-y|^{N+2s}}\,\di x\di y
\\
&&\hskip 7cm =\int_{\Omega}\frac{f_k(x)}{\left(u_k(x)+\frac{1}{k}\right)^{\gamma}}\mathcal{G}_{t,h}(u_k(x))\,\di x. \notag
\end{eqnarray}
We first deal with the left-hand side in \eqref{b}. All the arguments are contained in the proof of Theorem 3.1 in \cite{FV}, but we summarize them here for the reader's convenience.

\noindent We start by writing 
\begin{equation*}
\inte_{\R^{2N}}\frac{[u_k(x)-u_k(y)]\left[\mathcal{G}_{t,h}(u_k(x))-\mathcal{G}_{t,h}(u_k(y))\right]}{|x-y|^{N+2s}}\,\di x\di y
=\frac{1}{\Gamma\left(\frac{N+2s}{2}\right)}\int_0^{\infty} \mathcal{I}_{\alpha}[u_k,t,h]\alpha^{\frac{N+2s}{2}-1}\,\di\alpha,
\end{equation*}
where
$$
 \mathcal{I}_{\alpha}[u_k,t,h]=\inte_{\R^{2N}}\left[u_k(x)-u_k(y)\right]\left[\mathcal{G}_{t,h}(u_k(x) -\mathcal{G}_{t,h}(u_k(y) )\right]e^{-\alpha|x-y|^2}\,\di x\di y.
 $$
 
 \noindent Riesz inequality \eqref{riesz}, with the choices
 $$F(u_k, v_k)=u_k^2+v_k^2-(u_k-v_k)\left(\mathcal{G}_{t,h}(u_k)-\mathcal{G}_{t,h}(v_k)\right), \> W_{\alpha}(x)=e^{-\alpha |x|^2}, \>a=1, \> b=-1,$$
implies
 $$ \inte_{\R^{2N}} F(u_k(x),u_k(y))W_{\alpha}(x-y)\,\di x\di y\leq  \inte_{\R^{2N}} F(u_k^\star(x) ,u_k^\star(y))W_{\alpha}(x-y)\,\di x\di y,$$
 which immediately gives
$$ \mathcal{I}_{\alpha}[u_k,t,h]\geq  \mathcal{I}_{\alpha}[u_k^\star,t,h].$$
Hence
\begin{eqnarray}\label{b2}
&&\inte_{\R^{2N}}\frac{[u_k(x)-u_k(y)]\left[\mathcal{G}_{t,h}(u_k(x))-\mathcal{G}_{t,h}(u_k(y))\right]}{|x-y|^{N+2s}}\,\di x\di y 
\\
&&\hskip 5cm \ge \inte_{\R^{2N}}\frac{\left[u_k^\star(x)-u_k^\star(y)\right]\left[\mathcal{G}_{t,h}(u_k^\star(x))-\mathcal{G}_{t,h}(u_k^\star(y))\right]}{|x-y|^{N+2s}}\,\di x\di y. \notag 
\end{eqnarray}

\noindent In order to simplify the notation, from now on $\mathfrc{u}_k(x)=\mathfrc{u}_k(|x|)$ will stand for $u_k^\star(x)$.
We change the variables in the right-hand side of \eqref{b2} and we obtain
\begin{eqnarray}\label{b3}
&&\qquad\inte_{\R^{2N}}\frac{\left[u_k^\star(x)-u_k^\star(y)\right]\left[\mathcal{G}_{t,h}(u_k^\star(x))-\mathcal{G}_{t,h}(u_k^\star(y))\right]}{|x-y|^{N+2s}}\,\di x\di y
\\
&& =N \omega_N\int_0^{+\infty}\left(\int_0^{+\infty}[\mathfrc{u}_{k}(r)-\mathfrc{u}_{k}(\rho)]\left[\mathcal{G}_{t,h}(\mathfrc{u}_{k}(r))-\mathcal{G}_{t,h}(\mathfrc{u}_{k}(\rho))\right]\Theta_{N,s}(r,\rho)\rho^{N-1}\,\di \rho\right)r^{N-1}\,\di r,\notag
\end{eqnarray}
where $\Theta_{N,s}(r,\rho)$ is the function defined in \eqref{theta}.

\noindent We split the integral in the right-hand side of \eqref{b3} into the sum
$$\mathcal{I}^1+2\mathcal{I}^2+2\mathcal{I}^3+2h\mathcal{I}^4,$$
where
\begin{eqnarray*} 
\mathcal{I}^1&=&\int_{r(t+h)}^{r(t)}\left(\int_{r(t+h)}^{r(t)}(\mathfrc{u}_{k}(r)-\mathfrc{u}_{k}(\rho))^2\Theta_{N,s}(r,\rho)\rho^{N-1}\,\di \rho\right)r^{N-1}\,\di r,
\\
\mathcal{I}^2&=&\int_{0}^{r(t+h)}\left(\int_{r(t+h)}^{r(t)}(\mathfrc{u}_{k}(r)-\mathfrc{u}_{k}(\rho))(h-\mathfrc{u}_{k}(\rho)+t)\Theta_{N,s}(r,\rho)\rho^{N-1}\,\di \rho\right)r^{N-1}\,\di r,
\\
\mathcal{I}^3&=&\int_{r(t)}^{+\infty}\left(\int_{r(t+h)}^{r(t)}(\mathfrc{u}_{k}(r)-\mathfrc{u}_{k}(\rho))(-\mathfrc{u}_{k}(\rho)+t)\Theta_{N,s}(r,\rho)\rho^{N-1}\,\di \rho\right)r^{N-1}\,\di r,
\\
\mathcal{I}^4&=&\int_{0}^{r(t+h)}\left(\int_{r(t)}^{+\infty}(\mathfrc{u}_{k}(r)-\mathfrc{u}_{k}(\rho))\Theta_{N,s}(r,\rho)\rho^{N-1}\,\di \rho\right)r^{N-1}\,\di r,
\end{eqnarray*}
with $\mathfrc{u}_{k}(r(t))=t$ and $\mathfrc{u}_{k}(r(t+h))=t+h$.

\noindent Concerning the integral $\mathcal{I}^1$ we observe that, since $\mathfrc{u}$ is decreasing along the radii, we get
$$|\mathfrc{u}_{k}(r)-\mathfrc{u}_{k}(\rho)|\leq \mathfrc{u}_{k}(r(t+h))-\mathfrc{u}_{k}(r(t))=h.$$
Recalling the asymptotic behavior of  $\Theta_{N,s}(r,\rho)$ as $|\rho-r|\rightarrow 0$ given in \eqref{asym}, and that $\mathfrc{u}_{k}$ is $C^{s}(\R^{N})$ (see \cite[Proposition 1.1]{RS0})
the integral $\mathcal{I}^1$ can be estimated in the following way
$$\mathcal{I}^1\leq C h \int_{r(t+h)}^{r(t)}\left(\int_{r(t+h)}^{r(t)}|r-\rho|^{-1-s}\rho^{N-1}\,\di \rho\right)r^{N-1}\,\di r,$$
being $C$ a positive constant.
It follows that
\begin{equation}\label{l1}
\frac{\mathcal{I}^1}{h}\rightarrow 0,\quad\mbox{as}\,\,h\rightarrow 0^+.
\end{equation}
Similarly, we have

$$\mathcal{I}^2\leq  C h \int_{0}^{r(t+h)}\left(\int_{r(t+h)}^{r(t)}|\rho-r|^{-1-s}\rho^{N-1}\,\di \rho\right)r^{N-1}\,\di r,$$
which implies that
\begin{equation}\label{l2}
\frac{\mathcal{I}^2}{h}\rightarrow 0,\quad\mbox{as}\,\,h\rightarrow 0^+.
\end{equation}
We now consider the integral $\mathcal{I}^3$ and we observe that
\begin{eqnarray*}
\mathcal{I}^3&\le& \int_{r(t)}^{R}\left(\int_{r(t+h)}^{r(t)}|t-\mathfrc{u}_{k}(\rho)||r-\rho|^{-1-s}\rho^{N-1}\,\di \rho\right)r^{N-1}\,\di r
\\
&&+\, \int_{R}^{+\infty}\left(\int_{r(t+h)}^{r(t)}\mathfrc{u}_{k}(\rho)\,|t-\mathfrc{u}_{k}(\rho)|\Theta_{N,s}(\rho,r) \rho^{N-1}\,\di \rho\right)r^{N-1}\,\di r=\mathcal{I}^3_1+\mathcal{I}^3_2.
\end{eqnarray*}
For $\mathcal{I}^3_1$ we get
\begin{equation*}
\mathcal{I}^3_1\leq C h \int_{r(t)}^{R}\left(\int_{r(t+h)}^{r(t)}|r-\rho|^{-1-s}\,\di \rho\right)\,\di r
=Ch\left((R-r(t))^{1-s}-(R-r(t+h))^{1-s}\right),
\end{equation*}
so that 
\begin{equation}\label{l3}
\frac{\mathcal{I}_1^3}{h} \to 0 \qquad \mbox{ as } h \to 0^+.
\end{equation}
Recalling that $\Theta_{N,s}(r,\rho)$ is a symmetric function and it has the asymptotic behavior described in \eqref{asym1} as $r \to +\infty$, for $\mathcal{I}^3_2$ we have
\begin{eqnarray*}
\mathcal{I}^3_2&\le& C h \int_{R}^{+\infty}\left(\int_{r(t+h)}^{r(t)}\Theta_{N,s}(\rho,r) \rho^{N-1}\,\di \rho\right)r^{N-1}\,\di r 
\\
&=& Ch \int_{r(t+h)}^{r(t)}\left(\int_{R}^{+\infty}\frac{1}{r^{N+2s}}\, r^{N-1}\,\di r\right)\rho^{N-1}\,\di \rho,
\end{eqnarray*}
so that $\mathcal{I}_2^3$ also satisfies
\begin{equation}\label{l4}
\frac{\mathcal{I}_2^3}{h} \to 0 \qquad \mbox{ as } h \to 0^+.
\end{equation}
Gathering \eqref{l1}, \eqref{l2}, \eqref{l3} and \eqref{l4}, from \eqref{b3} we deduce

\begin{eqnarray}\label{left}
&&\lim_{h\rightarrow 0^+} \frac{1}{h}\inte_{\R^{2N}}\frac{\left[u_k^\star(x)-u_k^\star(y)\right]\left[\mathcal{G}_{t,h}(u_k^\star(x))-\mathcal{G}_{t,h}(u_k^\star(y))\right]}{|x-y|^{N+2s}}\,\di x\di y 
\\
 &&\hskip 2cm = 2 N\omega_N \int_{0}^{r(t)}\left(\int_{r(t)}^{+\infty}(\mathfrc{u}_{k}(r)-\mathfrc{u}_{k}(\rho))\Theta_{N,s}(r,\rho)\rho^{N-1}\,\di \rho\right)r^{N-1}\,\di r. \notag
\end{eqnarray}
We now focus on the right-hand side of \eqref{b}. Since
\begin{eqnarray*}
&&\int_{\Omega}\frac{f_k(x)}{\left(u_k(x)+\frac{1}{k}\right)^{\gamma}}\mathcal{G}_{t,h}(u_k(x))\,\di x
\\
&&\hskip 1cm \leq h ||f||_{L^\infty(\Omega)} \int_{u_k(x)>t+h} \frac{1}{\left(u_k(x)+\frac 1 k \right)^\gamma}\,\di x + ||f||_{L^\infty(\Omega)}\int_{t < u_k(x)\le t+h} \frac{u_k(x)-t}{\left(u_k(x)+\frac 1 k \right)^\gamma}\,\di x,
\end{eqnarray*}
we immediately get
\begin{eqnarray}\label{right}
\qquad\lim_{h \to 0^+} \frac 1 h \int_{\Omega}\frac{f_k(x)}{\left(u_k(x)+\frac{1}{k}\right)^{\gamma}}\mathcal{G}_{t,h}(u_k(x))\,\di x &\le&  ||f||_{L^\infty(\Omega)} \int_{u_k(x)>t} \frac{1}{\left(u_k(x)+\frac 1 k \right)^\gamma}\,\di x
\\
&=&N \omega_N ||f||_{L^\infty(\Omega)} \int_0^{r(t)} \frac{1}{\left(\mathfrc{u}_{k}(r)+\frac 1 k \right)^\gamma}r^{N-1}\,\di r. \notag
\end{eqnarray}
Assembling \eqref{b}, \eqref{b2}, \eqref{b3}, \eqref{left} and \eqref{right}, we finally obtain that, for $0\le t<||u_k||_{L^\infty(\Omega)}$, the following inequality holds true
\begin{eqnarray*} 
&&\gamma(N,s)\int_0^{r(t)}\left(\int_{r(t)}^{+\infty}(\mathfrc {u}_k(r)-\mathfrc{u}_k(\rho))\Theta_{N,s}(r,\rho)\rho^{N-1}\,\di \rho\right)r^{N-1}\,\di r
\\
&& \hskip 7cm \le  ||f||_{L^\infty(\Omega)}\int_0^{r(t)} \frac{1}{\left(\mathfrc{u}_k(r)+\frac{1}{k}\right)^{\gamma}}r^{N-1}\,\di r. \notag
\end{eqnarray*}
Reasoning as in \cite{FV} we can actually prove that, for every $r \ge 0$,
\begin{eqnarray} \label{b5}
&&\gamma(N,s)\int_0^{r}\left(\int_{r}^{+\infty}(\mathfrc {u}_k(\tau)-\mathfrc{u}_k(\rho))\Theta_{N,s}(\tau,\rho)\rho^{N-1}\,\di \rho\right)\tau^{N-1}\,\di \tau
\\
&& \hskip 7cm \le  ||f||_{L^\infty(\Omega)}\int_0^{r} \frac{1}{\left(\mathfrc{u}_k(\tau)+\frac{1}{k}\right)^{\gamma}}\,\tau^{N-1}\,\di \tau. \notag
\end{eqnarray}

\vskip 0.2cm

\noindent\emph{Step 3. Symmetrized approximating problems}

\smallskip

Let $v$ be the solution to the symmetrized problem \eqref{probSim}. We denote by $v_k$ the solution to the following problem 
\begin{equation*}\label{pks}
\left\{
\begin{array}{ll}
(-\Delta)^s v_k=\dfrac{||f||_{L^\infty(\Omega)}}{\left(v_k+\frac{1}{k}\right)^{\gamma}}\quad&\mbox{in}\,\Omega^\star
\\
v_k>0&\mbox{in }\Omega^\star
\\
v_k=0& \mbox{on }\mathbb{R}^N\setminus\Omega^\star.
\end{array}
\right.
\end{equation*}
Due to the radial symmetry and the radial monotonicity (see \cite{FW}), the function $v_k$ is such that $v_k(x)=v_k^\star(x)$ and, using the notation $\mathfrc{v}_k(x)=\mathfrc{v}_k(|x|)=v_k^\star(x)$, we have
\begin{eqnarray}\label{b6}
&&\gamma(N,s)\int_0^{r}\left(\int_r^{+\infty}(\mathfrc{v}_k(\tau)-\mathfrc{v}_k(\rho))\Theta_{N,s}(\tau,\rho)\rho^{N-1}\,\di \rho\right)\tau^{N-1}\,\di \tau
\\
&&\hskip 6cm = ||f||_{L^\infty(\Omega)}\int_0^{r} \frac{1}{\left(\mathfrc{v}_k(\tau)+\frac{1}{k}\right)^{\gamma}}\,\tau^{N-1}\,\di\tau. \notag
\end{eqnarray}

\vskip 0.2cm

\emph{Step 4. Comparison result}

\smallskip

Taking the difference between \eqref{b5} and \eqref{b6} we get
\begin{eqnarray}\label{abs}
&&\int_0^{r}\left(\int_r^{+\infty}\left[\left(\mathfrc{u}_k(\tau)-\mathfrc{v}_k(\tau)\right)-\left(\mathfrc{u}_k(\rho)-\mathfrc{v}_k(\rho)\right)\right]\Theta_{N,s}(\tau,\rho)\,\rho^{N-1}\,\di\rho\right)\tau^{N-1}\,\di \tau
\\
&&\hskip 5cm \le  ||f||_{L^\infty(\Omega)} \int_0^{r} \left(\frac{1}{\left(\mathfrc{u}_k(\rho)+\frac{1}{k}\right)^{\gamma}}-\frac{1}{\left(\mathfrc{v}_k(\rho)+\frac{1}{k}\right)^{\gamma}}\right)\tau^{N-1}\,\di\tau. \notag
\end{eqnarray}
We want to prove that
\begin{equation}\label{compSimgamma1}
 \int_{0}^{r}\mathfrc{u}_k(\tau)\tau^{N-1}\,\di \tau\leq \int_{0}^{r}\mathfrc{v}_k(\tau)\tau^{N-1}\,\di\tau,\qquad r\geq 0.
 \end{equation}

\noindent At this point, our approach greatly differs from the one used in  the proof of \cite[Theorem 31]{FV}, which consists in the interpretation of the LHS of \eqref{compSimgamma1} as the difference of $N+2$ dimensional fractional Laplacian of the spherical mean functions of $\mathfrc{u}_k$, $\mathfrc{v}_k$. Indeed, we use now a qualitative contradiction argument based on Lemma \ref{MaxMin}. Suppose by contradiction that the function $\dint_0^{r} \left(\mathfrc{u}_k(\tau)-\mathfrc{v}_k(\tau)\right)\tau^{N-1}\,\di\tau$ has a positive maximum point at $\bar{r}\in (0,R]$, i.e.,
\begin{equation*}\label{max>0}
0<\int_0^{\bar{r}} \left(\mathfrc{u}_k(\tau)-v_k(\tau)\right)\tau^{N-1}\,\di\tau=\max_{r\in[0,R]}\int_0^{r} \left(\mathfrc{u}_k(\tau)-\mathfrc{v}_k(\tau)\right)\tau^{N-1}\,\di\tau.
\end{equation*}
We recall that the function $\Theta_{N,s}(\tau,\rho)$ is increasing with respect to $\tau$ for any fixed $\rho>\bar r$. Hence, Lemma \ref{MaxMin} provides that, for every $\rho>\bar r$,
\begin{equation}\label{max+}
\int_0^{\bar{r}}(\mathfrc {u}_k(\tau)-\mathfrc{v}_k(\tau))\Theta_{N,s}(\tau,\rho)\tau^{N-1}\,\di\tau>0.
\end{equation}
According to what we notice in Remark \ref{hk}, if $\int_0^{r} \left(\mathfrc{u}_k(\tau)-\mathfrc{v}_k(\tau)\right)\tau^{N-1}\,\di\tau$ has a point of positive maximum at $\bar r$, then $\bar r$  is a point of non-positive minimum for $\int_r^R \left(\mathfrc{u}_k(\tau)-\mathfrc{v}_k(\tau)\right)\tau^{N-1}\,\di\tau$. Hence, using what noticed in Remark \ref{hk1} and the fact  that
 $\Theta_{N,s}(\tau,\rho)$ is decreasing with respect to $\rho$ for any fixed $\tau<\bar r$, we get that, for every $\tau<\bar{r}$,
\begin{equation}\label{minNeg}
\int_{\bar{r}}^R(\mathfrc {u}_k(\rho)-\mathfrc{v}_k(\rho))\Theta_{N,s}(\tau,\rho)\rho^{N-1}\,\di\rho \leq 0.
\end{equation}
\noindent From \eqref{max+} and \eqref{minNeg} we immediately deduce that

\begin{eqnarray}\label{twostars}
&&\int_0^{\bar{r}}\left(\int_{\bar{r}}^{+\infty}(\mathfrc {u}_k(\tau)-\mathfrc{u}_k(\rho))\Theta_{N,s}(\tau,\rho)\rho^{N-1}\,\di\rho\right)\tau^{N-1}\,\di\tau-
\\
&&\hskip 3cm \int_0^{\bar{r}}\left(\int_{\bar{r}}^{+\infty}(\mathfrc{v}_k(\tau)-\mathfrc{v}_k(\rho))\Theta_{N,s}(\tau,\rho)\rho^{N-1}\,\di\rho\right)\tau^{N-1}\,\di\tau>0.\notag
\end{eqnarray}
On the other hand, using Lemma \ref{ab} with the choice $a=\mathfrc{u}_k(\tau)+\frac 1 k$ and $b=\mathfrc{v}_k(\tau)+\frac 1 k$, we get that
$$
\int_0^{\bar r} \left(\frac{1}{\left(\mathfrc{u}_k(\tau)+\frac 1 k \right)^\gamma}-\frac{1}{\left(\mathfrc{v}_k(\tau)+\frac 1 k\right)^\gamma}\right)\,\tau^{N-1}\,\di \tau \le 
\gamma \int_0^{\bar r} \frac{1}{\left(\mathfrc{u}_k(\tau)+\frac 1 k \right)^{\gamma+1}}\left(\mathfrc{v}_k(\tau)-\mathfrc{u}_k(\tau)\right)\tau^{N-1}\,\di\tau,
$$
the last integral being negative via Lemma \ref{MaxMin} since $\frac{1}{\left(\mathfrc{u}_k(\tau)+\frac 1 k \right)^{\gamma+1}}$ is a positive, increasing function. This implies
\begin{equation}\label{star}
\int_0^{\bar r} \left(\frac{1}{\left(\mathfrc{u}_k(\tau)+\frac 1 k \right)^\gamma}-\frac{1}{\left(\mathfrc{v}_k(\tau)+\frac 1 k\right)^\gamma}\right)\,\tau^{N-1}\,\di \tau<0.
\end{equation}
Finally \eqref{twostars} and \eqref{star} contradict \eqref{abs} at $r=\bar r$.

\vskip 0.2cm

\noindent\emph{Step 5. Passing to the limit as $k \to +\infty$}

\smallskip

In \cite{BDMP} the authors prove  that the sequences $ u_k, v_k $ are bounded in $X_0^s(\Omega)$, resp. $X_0^s(\Omega^\star)$. Hence, up to subsequences, $u_k, v_k$ converge to functions $u \in X_0^s(\Omega)$, resp. $v \in X_0^s(\Omega^\star)$, weakly in $X_0^s$, strongly in $L^p$ for any $p\in [1,2^*_s)$ and a.e. in $\Omega$, resp. $\Omega^\star$. Moreover, $u$, resp. $v$, are solutions to problems \eqref{P}, resp. \eqref{probSim}. Hence we can pass to the limit in \eqref{compSimgamma1} getting 
$$
\int_0^r \mathfrc{u}(\tau)\tau^{N-1}\,\di\tau \le \int_0^r \mathfrc{v}(\tau)\tau^{N-1}\,\di\tau, \qquad r \ge 0,
$$
where $\mathfrc{u}(x)=\mathfrc{u}(|x|)=u^\star(x)$ and $\mathfrc{v}(x)=\mathfrc{v}(|x|)=v^\star(x)$, which is equivalent to \eqref{compSim}.
\hfill$\square$

\begin{remark}\label{remgamma}
In the local case, an analogous comparison result is proved in \cite{BCT}. There the authors prove that
\begin{equation}\label{loc}
\int_{B_r(0)} \frac{1}{u^\star(x)^\gamma}\,\di x \ge \int_{B_r(0)} \frac{1}{v^\star(x)^\gamma}\,\di x
\end{equation}
and, consequently, by multiplying both the integrands by $u^\star(x)^\gamma v^\star(x)^\gamma$ and using property $(b)$ in Proposition \ref{Propconves}, they have
\begin{equation}\label{loca}
\int_{B_r(0)} u^\star(x)^\gamma\,\di x \le \int_{B_r(0)}  v^\star(x)^\gamma\,\di x.
\end{equation}
Actually, by Lemma \ref{lemmaab} applied with the choice $a=u^\star(x)$ and $b=v^\star(x)$, we get
$$
u^\star(x)^{\gamma+1}\left(\frac{1}{u^\star(x)^\gamma}-\frac{1}{v^\star(x)^\gamma}\right)\le \gamma\left(v^\star(x)-u^\star(x)\right)
$$
so that \eqref{loc} and property $(b)$ in Proposition \ref{Propconves} imply
$$
\int_{B_r(0)} u^\star(x) \,\di x \le \int_{B_r(0)}  v^\star(x) \,\di x,
$$
which provides a more precise comparison result with respect to \eqref{loca} when $\gamma>1$.
\end{remark}
\begin{remark}
When $\gamma=0$, problem \eqref{P} coincides with the one discussed in \cite{FV}. If in \eqref{b5} we replace the right-hand side with $\dint_0^r f^\star(\tau)\tau^{N-1}\,\di \tau$, the subsequent arguments apply in order to gain an alternative proof of  the mass concentration estimate  \eqref{compSim}. We stress that in this case there is no need of an approximation procedure.
\end{remark}


\section{An explicit comparison result: proof of Theorem \ref{mainTheorem2}}

Theorem \ref{mainTheorem1} allows us to compare the solution to problem \eqref{P} with the solution to a symmetrized problem having the same structure. As in the local case (see \cite{BCT}), it is possible to compare $u$ with the solution to a symmetrized problem whose solution can be explicitely computed. Such a comparison result is a key ingredient to prove further regularity results.
%
%
%

\medskip
\noindent{\sc Proof of Theorem \ref{mainTheorem2}.}\\
We consider the same sequence of approximating problems \eqref{pk} that we examined in the previous section and, for $0\le t<||u_k||_{L^\infty(\Omega)}$ and $h>0$, we first prove that
 $\phi=u_k^{\gamma}\mathcal{G}_{t,h}\left(u^{\gamma+1}_k\right)$ can be chosen as test function in  \eqref{weakapprox}. Indeed, we have that by the mean value theorem, the boundedness of $u_{k}$ and the fact that $\mathcal{G}_{t,h}(\theta)\leq \theta$ yield
 \begin{align*}
 &|\phi(x)-\phi(y)|\leq u_{k}(x)\left|\mathcal{G}_{t,h}\left(u_k(x)^{\gamma+1}\right)-\mathcal{G}_{t,h}\left(u_k(y)^{\gamma+1}\right)\right|+\mathcal{G}_{t,h}\left(u_k(y)^{\gamma+1}\right)
 |u_{k}(x)^{\gamma}-u_{k}(y)^{\gamma}|\\
 &\leq C|u_k(x)^{\gamma+1}-u_k(y)^{\gamma+1}|+u_k(y)^{\gamma+1}\,|u_{k}(x)^{\gamma}-u_{k}(y)^{\gamma}|\\
 &\leq C_{\gamma}|u_{k}(x)-u_{k}(y)|+\gamma\left[\max\left\{u_{k}(x),u_{k}(y)\right\}\right]^{2\gamma}|u_{k}(x)-u_{k}(y)|
 \end{align*}
and the claim follows from the fact that $u_{k}\in X_{0}^{s}(\Omega)$.\\
Now putting $\phi$ in the weak formulation \eqref{weakapprox} we have
\begin{align}
\frac{\gamma(N,s)}{2}&\inte_{\R^{2N}}\frac{[u_k(x)-u_k(y)]\left[u_{k}(x)^{\gamma}\mathcal{G}_{t,h}\left(u_k(x)^{\gamma+1}\right)-u_{k}(y)^{\gamma}\mathcal{G}_{t,h}\left(u_k(y)^{\gamma+1}\right)\right]}{|x-y|^{N+2s}}\,\di x\di y \label{1}
\\&= \int_{\Omega} \frac{f_k(x)}{\left(u_k(x)+\frac{1}{k}\right)^{\gamma}}\,u_k(x)^{\gamma}\mathcal{G}_{t,h}\left(u_k(x)^{\gamma+1}\right)\,\di x.  \notag
\end{align}
By  using \eqref{leonori} with the choices
\[
\Phi(\theta)= \theta^{\gamma+1},\quad \varphi(x)=\mathcal{G}_{t,h}\left(u_k(x)^{\gamma+1}\right)
\]
we can estimate the left hand side of \eqref{1}  as follows
\begin{align}
&\inte_{\R^{2N}}\frac{[u_k(x)-u_k(y)]\left[u_{k}(x)^{\gamma}\mathcal{G}_{t,h}\left(u_k(x)^{\gamma+1}\right)-u_{k}(y)^{\gamma}\mathcal{G}_{t,h}\left(u_k(y)^{\gamma+1}\right)\right]}{|x-y|^{N+2s}}\,\di x\di y \label{bb}\\
&\geq 
\frac{1}{\gamma+1}\inte_{\R^{2N}}\frac{\left[u_k(x)^{\gamma+1}-u_k(y)^{\gamma+1}\right]\left[\mathcal{G}_{t,h}\left(u_k(x)^{\gamma+1}\right)-\mathcal{G}_{t,h}\left(u_k(y)^{\gamma+1}\right)\right]}{|x-y|^{N+2s}}\,\di x\di y.  \notag
\end{align}

Reasoning as in the previous section, we can show that
\begin{eqnarray}\label{3}
&& \inte_{\R^{2N}}\frac{\left[u_k(x)^{\gamma+1}-u_k(y)^{\gamma+1}\right]\left[\mathcal{G}_{t,h}\left(u_k(x)^{\gamma+1}\right)-\mathcal{G}_{t,h}\left(u_k(y)^{\gamma+1}\right)\right]}{|x-y|^{N+2s}}\,\di x\di y
\\
&&\hskip 3cm\geq  \inte_{\R^{2N}}\frac{\left[u_k^\star(x)^{\gamma+1}-u_k^\star(y)^{\gamma+1}\right]\left[\mathcal{G}_{t,h}\left(u_k^\star(x)^{\gamma+1}\right)-\mathcal{G}_{t,h}\left(u_k^\star(y)^{\gamma+1}\right)\right]}{|x-y|^{N+2s}}\,\di x\di y \notag
\end{eqnarray}
and
\begin{eqnarray}
&&\lim_{h\rightarrow 0^+} \frac 1 h  \inte_{\R^{2N}}\frac{\left[u_k^\star(x)^{\gamma+1}-u_k^\star(y)^{\gamma+1}\right]\left[\mathcal{G}_{t,h}\left(u_k^\star(x)^{\gamma+1}\right)-\mathcal{G}_{t,h}\left(u_k^\star(y)^{\gamma+1}\right)\right]}{|x-y|^{N+2s}}\,\di x\di y
\\
&&\hskip 3cm= \int_{0}^{r(t)}\left(\int_{r(t)}^{+\infty}(\mathfrc{u}_k(r)^{\gamma+1}-\mathfrc{u}_k(\rho)^{\gamma+1})\Theta_{N,s}(r,\rho)\rho^{N-1}\,\di\rho\right)r^{N-1}\,\di r, \notag
\end{eqnarray}
where, as in the previous section, $\mathfrc{u}_k(x)=\mathfrc{u}_k(|x|)$ stands for $u_k^\star(x)$. 
Regarding the right-hand side of \eqref{1}, we observe that
\begin{eqnarray*}
&&\int_{\Omega} \frac{f_k(x)}{\left(u_k(x)+\frac{1}{k}\right)^{\gamma}}\,u_k(x)^{\gamma}\mathcal{G}_{t,h}\left(u_k(x)^{\gamma+1}\right)\,\di x
\\
&&=h \int_{u_k^{\gamma+1}>t+h} \frac{f_k(x)}{\left(u_k(x)+\frac{1}{k}\right)^{\gamma}}\,u_k(x)^{\gamma}\,\di x
+\int_{t<u_k^{\gamma+1}\leq t+h} \frac{f_k(x)}{\left(u_k(x)+\frac{1}{k}\right)^{\gamma}}\,u_k(x)^{\gamma}(u_k(x)^{\gamma+1}-t)\,\di x 
\end{eqnarray*}
and 
$$\int_{t<u_k^{\gamma+1}\leq t+h} \frac{f_k(x)}{\left(u_k(x)+\frac{1}{k}\right)^{\gamma}}\,u_k(x)^{\gamma}(u_k(x)^{\gamma+1}-t)\,\di x\leq h\,||f||_{\infty}\int_{t<u_k^{\gamma+1}\leq t+h} \di x.$$
It follows that
\begin{equation}\label{2}
\lim_{h \to 0^+} \frac{1}{h}\int_{\Omega} \frac{f_k(x)}{\left(u_k(x)+\frac{1}{k}\right)^{\gamma}}\,u_k(x)^{\gamma}\mathcal{G}_{t,h}\left(u_k(x)^{\gamma+1}\right)\,\di x=\int_{u_k^{\gamma+1}>t} \frac{f_k(x)}{\left(u_k(x)+\frac{1}{k}\right)^{\gamma}}\,u_k(x)^{\gamma}\,\di x.
\end{equation}
It is easy to observe that
\begin{equation}\label{bbb}
\int_{u_k^{\gamma+1}>t} \frac{f_k(x)}{\left(u_k(x)+\frac{1}{k}\right)^{\gamma}}\,u_k(x)^{\gamma}\,\di x\leq \int_{u_k^{\gamma+1}>t} f(x)\,\di x\leq \int_0^{r(t)}f^\star(\rho)\rho^{N-1}\,\di \rho,
\end{equation}
being $\mathfrc{u}_k(r(t))^{\gamma+1}=t$.
From \eqref{1}-\eqref{bbb} we deduce
$$\frac{\gamma(N,s)}{2(\gamma+1)} \int_{0}^{r(t)}\left(\int_{r(t)}^{+\infty}(\mathfrc{u}_k(r)^{\gamma+1}-\mathfrc{u}_k(\rho)^{\gamma+1})\Theta_{N,s}(r,\rho)\rho^{N-1}\,\di \rho\right)r^{N-1}\,\di r \leq  \int_0^{r(t)}f^\star(\rho)\rho^{N-1}\,\di \rho.$$
Reasoning as in \cite{FV} we can show that actually the following inequality holds true for every $r\ge 0$:
\begin{equation}\label{u}
\frac{\gamma(N,s)}{2(\gamma+1)} \int_{0}^{r}\left(\int_{r}^{+\infty}(\mathfrc{u}_k(\tau)^{\gamma+1}-\mathfrc{u}_k(\rho)^{\gamma+1})\Theta_{N,s}(\tau,\rho)\rho^{N-1}\,\di\rho\right)\tau^{N-1}\,\di\tau \leq \int_0^r f^\star(\rho)\rho^{N-1}\,\di \rho.
\end{equation}
On the other hand, the solution to problem \eqref{probSimgamma} satisfies
\begin{equation}\label{v}
\frac{\gamma(N,s)}{2(\gamma+1)} \int_{0}^{r}\left(\int_{r}^{+\infty}(\mathfrc{v}(\tau)-\mathfrc{v}(\rho))\Theta_{N,s}(\tau,\rho)\rho^{N-1}\,\di\rho\right)\tau^{N-1}\,\di\tau = \int_0^r f^\star(\rho)\rho^{N-1}\,\di \rho.
\end{equation}
Subtracting \eqref{u} and \eqref{v},we can conclude as in the proof of Theorem \ref{mainTheorem1}.


\section{Some regularity results}

As an immediate consequence of Theorem \ref{mainTheorem2} we can prove the following regularity results, depending on the value of $\gamma$ and on the summability of $f$.

\begin{theorem}\label{estimates}
Let $s \in (0,1)$, $N \ge 2$, $\gamma>0$, and assume that $f \in L^p(\Omega)$, with $p\geq 2^{\ast}_{s}$, $f \ge 0$. If $u \in X_0^s(\Omega)$ is the weak solution to problem \eqref{P}, the following estimates hold true.
\begin{enumerate}
    \item If $p<\frac{N}{2s}$, then $u \in L^q(\Omega)$, with $q=\frac{Np(\gamma+1)}{N-2sp}$, and there exists a positive constant $C$ such that
    $$
    ||u||_{L^q(\Omega)}\le C ||f||^{1/(\gamma+1)}_{L^p(\Omega)}.
    $$
    \item If $p=\frac{N}{2s}$, then $u \in L_{\Phi}(\Omega)$, where $L_{\Phi}(\Omega)$ is the Orlicz space generated by the $N$-function
    $$
    \Phi(t)=\exp(|t|^{(\gamma+1)p'})-1.
    $$
    Moreover, there exists a positive constant $C$ such that
    $$
    ||u||_{L_{\Phi}(\Omega)}\le C ||f||_{L^p(\Omega)}^{1/(\gamma+1)}.
    $$
    \item If $p>\frac{N}{2s}$, then $u \in L^\infty(\Omega)$ and there exists a positive constant $C$ such that
    $$
    ||u||_{L^\infty(\Omega)}\le C ||f||_{L^p(\Omega)}^{1/(\gamma+1)}.
    $$
\end{enumerate}
\end{theorem}

\begin{proof}
We simply observe that for $q>\gamma+1$ we have
\[
\|u\|_{L^{q}(\Omega)}=\|u^{\gamma+1}\|_{\frac{q}{\gamma+1}}^{1/(\gamma+1)},
\]
therefore by  Theorem \ref{mainTheorem2}
\[
\|u\|_{L^{q}(\Omega)}\leq \|v\|_{L^{q}(\Omega^{\star})}^{1/(\gamma+1)}.
\]
Then we can apply the regularity result \cite[Theorem 3.2]{FV} to the solution $v$ to the symmetrized problem \eqref{probSimgamma}. Moreover, in the limit case $p=N/2s$ we notice that Theorem \ref{mainTheorem2} implies
\[
(u^{\gamma+1})^{\ast\ast}\le v^{\ast\ast}
\]
and arguing as in the proof of \cite[Theorem 3.2]{FV} the claim follows.

\end{proof}


\begin{remark}
We stress that when $\gamma=0$ we recover the estimates contained in \cite{FV}, while when $s=1$ we have the same estimates contained in \cite{BO,BCT}.
\end{remark}

We end the paper with the following energy estimate.
\begin{proposition}
Let $s\in(0,1), N\geq 2, \gamma>0$ and  assume that $f\in L^{(2^*_s)'}(\Omega), f\geq 0$. If $u\in X_0^s(\Omega)$ and $v \in X_0^s(\Omega^\star)$ are the weak solutions to problems \eqref{P} and \eqref{probSimgamma}, respectively, then 

\begin{equation*}\label{confrontoEnergiegamma+1}
||u^{\gamma+1}||_{X_0^s(\Omega)}\leq ||v||_{X_0^s(\Omega^\star)}.
\end{equation*}

\end{proposition}

\proof
Let $k \in \N$ and let $u_k$ be a solution to \eqref{pk}. Let $T>1$. We consider the following function $\Phi_{T}:[0,+\infty)\rightarrow [0,+\infty)$ defined as

$$\Phi_{T}(\theta)=
 \begin{cases}
\theta^{\gamma+1}\quad\qquad\quad\qquad\quad\quad\mbox{if}\,\,0\leq \theta<T\\
(\gamma+1)T^{\gamma}\theta-\gamma \,T^{\gamma+1}\quad\mbox{if}\,\,\theta\geq T.
\end{cases}
$$
Since $\Phi_{T}(\theta)$ and $\Phi_{T}(\theta)\Phi'_{T}(\theta)$ are  Lipschitz continuous functions, $\Phi_{T}(u_{k})$ and $\Phi_{T}(u_{k})\Phi'_{T}(u_{k})$ belong to $X_{0}^s(\Omega)$. Inequality \eqref{leonori} with the choice $\varphi=\Phi_{T}(u_{k})$ implies
\begin{align}\label{Rgamma+1}
|| \Phi_{T}(u_k)&||^2_{X_0^s(\Omega))}=\inte_{\R^{2N}}\frac{|\Phi_{T}(u_{k}(x))-\Phi_{T}(u_{k}(y))|^{2}}{|x-y|^{N+2s}}\,\di x\di y 
\\
&\leq\inte_{\R^{2N}}\frac{\left[u_{k}(x)-u_{k}(y)\right]\left[\Phi_{T}^{\prime}(u_{k}(x)\Phi_{T}(u_{k}(x))-\Phi_{T}^{\prime}(u_{k}(y))\Phi_{T}(u_{k}(y))\right]}{|x-y|^{N+2s}}\,\di x\di y
\notag
\\
&=\frac{2}{\gamma(N,s)} \int_{\Omega}\frac{f_k(x)}{(u_k(x)+\frac{1}{k})^{\gamma}}\Phi_{T}'(u_k(x))\Phi_{T}(u_k(x))\,\di x \notag
\\
&\leq \frac{2 (\gamma+1)}{\gamma(N,s)}\int_{\Omega}\frac{f_k(x)}{(u_k(x)+\frac{1}{k})^{\gamma}} u_k(x)^{\gamma+1} u_k(x)^{\gamma}\,\di x \notag
\\
&\leq \frac{2 (\gamma+1)}{\gamma(N,s)}\int_{\Omega}f(x) u_k(x)^{\gamma+1}\,\di x. \notag
\end{align}
Recall now that the proof of Theorem \ref{mainTheorem2} gives \eqref{compSim2} being $u$ replaced by $u_{k}.$ Then using the Hardy-Littlewood  inequality \eqref{hl} and Proposition \ref{Propconves}, we can estimate the right hand side of \eqref{Rgamma+1} as follows
$$\int_{\Omega}f (x)u_k(x)^{\gamma+1}\,\di x\leq \int_{\Omega^\star}f^\star(x)v(x)\,\di x,$$
then from \eqref{Rgamma+1} we conclude

\begin{equation}\label{enfin}
|| \Phi_T(u_k)||^2_{X_0^s(\Omega)}\leq \frac{2 (\gamma+1)}{\gamma(N,s)}\int_{\Omega^\star}f^\star(x)v(x)\,\di x=||v||_{X_0^s(\Omega^\star)}^{2}.
\end{equation}
Estimate \eqref{enfin} implies that the family $\Phi_{T}(u_{k})$ is uniformly bounded with respect to $T>1$ and $k\in\mathbb{N}$. Consequently, by the Sobolev embedding theorem we can extract a subsequence $T_{\ell}\rightarrow+\infty$ such that
\[
\Phi_{T_{\ell}}(u_{k})\rightharpoonup u_{k}^{\gamma+1}\quad \mbox{weakly in }X_{0}^{s}(\Omega),\quad \Phi_{T_{\ell}}(u_{k})\rightarrow u_{k}^{\gamma+1}\quad \mbox{stronlgy in }L^q(\Omega), \> q<2_{s}^{\ast},
\] 
as $\ell\rightarrow +\infty$. Then we can pass to the limit in \eqref{enfin} and obtain
\begin{equation*}
||u_k^{\gamma+1}||_{X_0^s(\Omega)}\leq ||v||_{X_0^s(\Omega^\star)}.
\end{equation*}

\noindent Thanks to the lower semicontinuity of the norm, we can pass to the limit in the previous inequality as $k$ goes to $+\infty$ and we get the claim.
\endproof

%
%
%
%
\section*{Acknowledgments}

All authors were partially supported by Italian MIUR through research project PRIN 2017 ``Direct and inverse problems for partial differential equations: theoretical aspects and applications'' and by Gruppo Nazionale per l'Analisi Matematica, la Probabilit\`a e le loro Applicazioni (GNAMPA) of Istituto Nazionale di Alta Matematica (INdAM). B.V. wishes to warmly thank L. Brasco for fruitful discussions.



\medskip
{\footnotesize
\noindent (Barbara Brandolini) Dipartimento di Matematica e Informatica, Universit\`a degli Studi 
di Palermo, 
via Archirafi 34, 90123 Palermo, Italy, e-mail: \texttt{barbara.brandolini@unipa.it}}

 \medskip
{\footnotesize
\noindent (Ida de Bonis) Dipartimento di Pianificazione, Design, Tecnologia dell'Architettura, Sapienza Universit\`a di Roma, via Flaminia 72, 00196 Roma, Italy, e-mail: \texttt{ida.debonis@uniroma1.it}}

\medskip
{\footnotesize 
\noindent (Vincenzo Ferone) Dipartimento di Matematica e Applicazioni ``R. Caccioppoli'', Universit\`a degli Studi 
di
Napoli Federico II, Complesso Univ. Monte S. Angelo, via Cintia, 80126 Napoli,
Italy, e-mail: \texttt{ferone@unina.it}} 
\medskip

{\footnotesize
\noindent (Bruno Volzone) Dipartimento di Scienze Economiche, Giuridiche, Informatiche e Motorie, Universit\`a degli Studi 
di Napoli ``Parthenope'', Via Guglielmo Pepe, Rione Gescal, 80035 Nola, Italy, e-mail: \texttt{bruno.volzone@uniparthenope.it}}


\begin{thebibliography}{20}

\bibitem{AL} {\sc F. J. Almgren, and E. H. Lieb}, \emph{Symmetric decreasing rearrangement is sometimes continuous}, J. Am. Math. Soc. 2 (1989), 683--773.


\bibitem{ALT} {\sc A. Alvino, P. L. Lions, and G. Trombetti}, \emph{On optimization problems with prescribed rearrangements}, Nonlinear Anal. Theory Methods Appl. 13 (1989), 185--220.

\bibitem{AVV} {\sc A. Alvino, R. Volpicelli, and B. Volzone}, \emph{Sharp estimates for solutions of parabolic equations with a lower order term}, J. Appl. Funct. Anal. 3  (2008) no. 1, 61--88.

\bibitem{Bandle} {\sc C. Bandle},  Isoperimetric inequalities and applications, Monographs and Studies
in Mathematics, vol. 7. , Pitman (Advanced Publishing Program), Boston, MA 1980.


\bibitem{BDMP} {\sc B. Barrios, I. de Bonis, M. Medina, and I. Peral}, {\em Semilinear problems for the fractional laplacian with a singular nonlinearity}, Open Mathematics 13 (2015),  390--407.

\bibitem{BS} {\sc C. Bennett,  and R. Sharpley},  Interpolation of operators,  Pure and Applied Mathematics, vol. 129, Academic Press Inc., Boston, MA 1988.

\bibitem{BO} {\sc L. Boccardo, and L. Orsina}, \emph{Semilinear elliptic equations with singular nonlinearities}, Calc. Var. Partial Differ. Equ. 37 (2010), 363--380.

\bibitem{SirBonfVaz}
{\sc M.~Bonforte, Y.~Sire, and J.~L. V{\'a}zquez}, {\em Existence, uniqueness
  and asymptotic behaviour for fractional porous medium equations on bounded
  domains}, Discrete Contin. Dyn. Syst. 35 (2015), 5725--5767.


\bibitem{BCT} {\sc B. Brandolini, F. Chiacchio, and C. Trombetti}, \emph{Symmetrization for singular semilinear elliptic equations}, 
Ann. Mat. Pura Appl. (4) 193 (2014), no. 2, 389--404.

\bibitem{Brasco}
{\sc L.~Brasco, E.~Lindgren, and E.~Parini}, {\em The fractional {C}heeger
  problem}, Interfaces and Free Boundaries 16 (2014), 419--458.


\bibitem{secondeigenbrasco}
{\sc L.~{Brasco} and E.~{Parini}}, {\em {The second eigenvalue of the
  fractional $p$-Laplacian}}, {Adv. Calc. Var.} 9 (2016), 323--355.

\bibitem{BrasParSquas}
{\sc L.~{Brasco}, E.~{Parini}, and M.~{Squassina}}, {\em {Stability of
  variational eigenvalues for the fractional $p$-Laplacian}}, {Discrete Contin.
  Dyn. Syst.} 36 (2016), 1813--1845.


%
%

\bibitem{CMSS} {\sc A. Canino, L. Montoro, B. Sciunzi and M. Squassina}, \emph{Nonlocal problems with singular nonlinearity}, Bull. Sci. Math. 141 (2017), 223--250.

\bibitem{DV} {\sc G. di Blasio, and B. Volzone},\emph{Comparison and regularity results for the fractional Laplacian via Symmetrization
methods}, Journal of Differential Equations 253 (2012), 2593--2615.

\bibitem{DNPALVAL} {\sc E. Di Nezza, G. Palatucci, and E. Valdinoci},  {\em Hitchhiker's guide to the fractional Sobolev spaces}, Bull. Sci. Math.  136 (5) (2012),
521--573.

\bibitem{FW} {\sc P. Felmer, and Y. Wang}, {\em Radial symmetry of positive solutions to equations involving
the fractional Laplacian}, Commun. Contemp. Math. 16 (1) (2014), 1350023, pp. 24. 

\bibitem{FSV} {\sc F. Feo, P. R. Stinga, and B. Volzone}, \emph{The fractional nonlocal Ornstein-Uhlenbeck equation, Gaussian symmetrization and regularity}, Discrete Contin. Dyn. Syst. 38 (2018), 3269--3298.

\bibitem{FV}  {\sc V. Ferone, and B. Volzone}, {\em Symmetrization for fractional elliptic problems: a direct approach}, Arch. Rational Mech. Anal.  239 (2021),  1733--1770.

\bibitem{FVnon}  {\sc V. Ferone, and B. Volzone}, {\em 
Symmetrization for fractional nonlinear elliptic problems}
Discrete and Continuous Dynamical Systems - A (2022).

\bibitem{GON}  {\sc G. Galiano}, {\em Symmetrization in nonlocal diffusion problems}, arXiv:2208.14735 (2022).

\bibitem{Kes} {\sc S. Kesavan}, Symmetrization \& applications, Series in Analysis, vol. 3, World Scientific Publishing Co. Pte. Ltd., Hackensack, NJ 2006.


\bibitem{L}  {\sc N. N. Lebedev}, Special functions and their applications. Revised edition, translated from the Russian and edited by Richard A. Silverman. Unabridged and corrected republication. Dover Publications, Inc., New York, 1972.
 
\bibitem{LPPS} {\sc T. Leonori, I. A. Peral, A. Primo, and F. Soria}, {\em Basic estimates for solutions of a class of nonlocal elliptic and parabolic equations}, Discrete and Continuous Dynamical Systems  35 (12) (2015),  6031--6068.

%

\bibitem{RS0} {\sc X. Ros-Oton and J. Serra}, \emph{The Dirichlet problem for the fractional Laplacian: regularity up to the boundary}, J. Math. Pures Appl. (9) 101, No. 3 (2014), 275--302.

%
%
\bibitem{SVV} {\sc Y. Sire, J. L. V{\'a}zquez, and B. Volzone}, \emph{Symmetrization for fractional elliptic and parabolic equations and an isoperimetric application}, Chin. Ann. Math. Ser. B 38 (2017), 661--686.

\bibitem{S} {\sc E. M. Stein}, Singular Integrals and Differential Properties of Functions, Princeton Mathematical Series, No. 30, Princeton University Press, Princeton, N. J. 1970.


\bibitem{T} {\sc G. Talenti}, \emph{Elliptic equations and rearrangements}, Ann. Scuola Norm. Sup. Pisa Cl. Sci. (4) 3 (1976), 697--718.

\bibitem{Talenti} {\sc G. Talenti}, {\em Inequalities in rearrangement invariant function spaces}, In: Nonlinear analysis, function spaces and applications, Vol. 5 (Prague, 1994), 177--230. Prometheus, Prague 1994.

\bibitem{Ta} {\sc G. Talenti}, \emph{The art of rearranging},  Milan J. Math. 84 (1) (2016), 105--157. 

%

\bibitem{VV1} {\sc J. L. V{\'a}zquez, and B. Volzone}, \emph{Symmetrization for linear and nonlinear fractional parabolic equations of porous medium type}, J. Math. Pures Appl. 101 (9) (2014), 553--582.

\bibitem{VV2} {\sc J. L. V{\'a}zquez, and B. Volzone}, \emph{Optimal estimates for fractional fast diffusion equations}, J. Math. Pures Appl. 103 (9) (2015), 535--556.

 \bibitem{V} {\sc B. Volzone}, \emph{Symmetrization for fractional Neumann problems}, Nonlinear Anal. 147 (2016), 1--25.
    


    
    

\end{thebibliography}
\end{document}